\documentclass[12pt]{amsart}
\usepackage{amsfonts}
\usepackage{graphicx}
\usepackage{amsmath}
\usepackage{amssymb}
\input{amssym.def}
\newtheorem{theorem}{Theorem}
\newtheorem{proposition}{Proposition}
\newtheorem{corollary}{Corollary}
\newtheorem{lemma}{Lemma}
\newtheorem{remark}{Remark}

\theoremstyle{remark}
\newtheorem{example}{Example}
\newcommand{\re}{\text{\rm Re }}
\newcommand{\im}{\text{\rm Im }}

\input epsf
\begin{document}
\title[Two-functional conjecture]{Sufficient conditions in the two-functional conjecture for
univalent functions}
\author[D.~Prokhorov]{Dmitri Prokhorov}

\dedicatory{Dedicated to Professor Promarz M. Tamrazov}

\address{D.~Prokhorov: Department of Mathematics and Mechanics, Saratov State University, Saratov
410012, Russia} \email{ProkhorovDV@info.sgu.ru}

\begin{abstract}
The two-functional conjecture says that if a function $f$ analytic and univalent in the unit disk
maximizes $\re\{L\}$ and $\re\{M\}$ for two continuous linear functionals $L$ and $M$, $L\neq cM$
for any $c>0$, then $f$ is a rotation of the Koebe function. We use the L\"owner differential
equation to obtain sufficient conditions in the two-functional conjecture and compare the
sufficient conditions with necessary conditions.
\end{abstract}
\maketitle

\section{Introduction}

For the class $S$ of functions $f(z)=z+a_2z^2+\dots$ analytic and univalent in the unit disk
$\mathbb D$, the two-functional conjecture arose from the description of functions $f\in S$ which
satisfy two independent so-called $\mathcal D_n$-equations, see, e.g. [8, p.347-351] and references
therein. This conjecture says that if a function $f\in S$ maximizes $\re\{L\}$ and $\re\{M\}$ for
two continuous linear functionals $L$ and $M$ nonconstant on $S$, $L\neq cM$ for any $c>0$, then
$f$ is a rotation of the Koebe function $k(z)=z(1-z)^{-2}$.

Each continuous linear functional on the space $\mathcal A$ of all analytic functions in $\mathbb
D$ has the form $L(h)=\sum_{n=0}^{\infty}c_na_n$, $h(z)=\sum_{n=0}^{\infty}a_nz^n$, for some
sequence of complex numbers $c_n$, $\lim\sup_{n\to\infty}|c_n|^{1/n}<1$, see, e.g. [8, p.280]. The
known results \cite{Bakhtin-1}, \cite{Bakhtin-2}, \cite{Goh}, \cite{Starkov}, see also survey
\cite{Prokhorov1}, for special cases of the two-functional conjecture are restricted to functionals

\begin{equation}
L(f)=\sum_{n=2}^n\overline{\lambda}_ka_k,\;\;\lambda_n\neq0,\;\;\text{and}\;\;
M(f)=\sum_{n=2}^m\overline{\mu}_ka_k,\;\;\mu_m\neq0. \label{fun}
\end{equation}
We shall be assuming $m=n$, since the conjecture is true for the case $m\neq n$ if it is proved for
$m=n$, \cite{Goh}.

The two-functional conjecture if it is true characterizes an exclusive role of the Koebe function
and its rotations both analytically and geometrically. After de Branges \cite{Branges} proved the
Bieberbach conjecture it became clear that the Koebe function $k(z)$ maximizes simultaneously
$\re\{L_j\}$ for $n-1$ independent continuous linear functionals $L_k=a_k+a_n$, $k=2,\dots,n-1$,
and $L_n=a_n$. This means geometrically that $k(z)$ delivers a boundary point $x^0=(2,\dots,n)$ to
the value set
$$V_n=\{(a_2,\dots,a_n): f\in S\},$$ and there is at least $(n-1)$-dimensional set of support
hyperplanes for the $(2n-2)$-dimensional set $V_n$ through $x^0$, namely, hyperplanes with normal
vectors $\Lambda_n=(0,\dots,0,1)$ and $\Lambda_k=(0,\dots,0,1,0,\dots,0,1)$, the unit is at the
$(k-1)$-th place. The result due to Bshouty and Hengartner \cite{Bshouty} shows that the set of
support hyperplanes for $V_n$ through $x^0$ is exactly $(n-1)$-dimensional since if at least one of
coefficients $\lambda_k$ in (\ref{fun}) is not real, then $k(z)$ does not maximize $\re\{L\}$.

So the two-functional conjecture supposes that if a function $f\in S$ maximizes $\re\{L\}$ and $f$
is different from any rotation of $k(z)$, then there is only one support hyperplane for $V_n$
through a boundary point $x_f\in\partial V_n$ delivered by $f$.

In the present article we give sufficient conditions for the two-functional conjecture in terms of
coefficients $\overline{\lambda}_k$, $\overline{\mu}_k$, $k=2,\dots,n$, and coefficients of an
extremal function $f\in S$. We prove the following theorem.

\begin{theorem}
Let a function $f(z)=z+a_2z^2_+\dots$ maximize $\re\{L\}$ and $\re\{M\}$ on $S$ where $L$ and $M$
are given by (\ref{fun}), $m=n$. Suppose that the trigonometric polynomial
\begin{equation}
\re\left(-\sum_{k=2}^n\sum_{j=1}^{n-k+1}\overline{\lambda}_{j+k-1}ja_je^{-i(k-1)u}\right)
\label{th1}
\end{equation}
attains its maximum on $[0,2\pi]$ at $u=\pi$, and
\begin{equation}
\sum_{k=2}^n\sum_{j=1}^{n-k+1}(-1)^k(k-1)^2\;\re(((1-\alpha)\overline{\lambda}_{j+k-1}+
\alpha\overline{\mu}_{j+k-1})ja_j)\neq0,\;\;\alpha\in[0,1]. \label{th2}
\end{equation}
Then
\begin{equation}
\sum_{k=2}^n\sum_{j=1}^{n-k+1}(-1)^k(k-1)\im(\overline{\lambda}_{j+k-1}ja_j)=
\sum_{k=2}^n\sum_{j=1}^{n-k+1}(-1)^k(k-1)\im(\overline{\mu}_{j+k-1}ja_j). \label{th3}
\end{equation}
If additionally
\begin{equation}
\sum_{k=2}^n\sum_{j=1}^{n-k+1}(-1)^k(k-1)^{m+1}\im(\overline{\lambda}_{j+k-1}ja_j)= \label{th4}
\end{equation}
$$\sum_{k=2}^n\sum_{j=1}^{n-k+1}(-1)^k(k-1)^{m+1}\im(\overline{\mu}_{j+k-1}ja_j)=0,\;\;\;
m=1,2,\dots,$$ then $f(z)$ is the Koebe function $k(z)$.
\end{theorem}

In Section 2 we introduce elements of the L\"owner theory necessary for the proof method. We note
there that every extremal function obeys condition (\ref{th1}) of Theorem 1 up to a suitable
rotation. This notion is fixed in Remark 1 of Section 3.

Theorem 1 is proved in Section 3 where we remark also that condition (\ref{th2}) of Theorem 1 is
not essential.

Section 4 contains necessary conditions in the two-functional conjecture.

\section{The L\"owner interpretation of an extremal problem}

A function $f\in S$ is a support point of $S$ if there is a continuous linear functional $L$, not
constant on $S$, such that $f$ maximizes $\re\{L\}$ over $S$. Every support point of $S$ maps
$\mathbb D$ onto the complement of a single analytic arc extending from a finite point to infinity,
see [12, p.149], \cite{Brickman}, [8, p.306].

From the other side, a function $f\in S$ which maps $\mathbb D$ onto the plane $\mathbb C$ minus a
single slit $\Gamma$ can be represented as
\begin{equation}
f(z)=\lim\limits_{t\to\infty}e^tw(z,t),\quad z\in \mathbb D, \label{lim}
\end{equation}
where $w(z,t)=e^{-t}(z+a_2(t)z^2+\dots)$ is a solution to the equation
\begin{equation}
\frac{dw}{dt}=-w\frac{e^{iu(t)}+w}{e^{iu(t)}-w}, \quad w(z,0)\equiv z. \label{leo}
\end{equation}

The driving term $u(t)$ in the L\"owner ordinary differential equation (\ref{leo}) is real analytic
provided the slit $\Gamma$ is analytic \cite{Earle}, see also \cite{Marshall}. Let us express an
extremal function $f$ for functionals (\ref{fun}) in terms of an optimal driving function $u$ for
the system of differential equations generated by the L\"owner equation (\ref{leo}). Put

$$
a(t)=\left(\begin{matrix}
a_1(t)\\
a_2(t)\\
\vdots\\
a_n(t)
\end{matrix}\right),\;\;
a^0=\left(\begin{matrix}
1\\
0\\
\vdots\\
0
\end{matrix}\right),\;\;
A(t)=\left(\begin{matrix}
0 & 0 & \hdots & 0 & 0 \\
a_1(t) & 0 & \hdots & 0 & 0 \\
a_2(t) & a_1(t) & \hdots & 0 & 0 \\
\vdots & \vdots & \ddots & \vdots & \vdots \\
a_{n-1}(t) & a_{n-2}(t) & \hdots & a_1(t) & 0
\end{matrix}\right).
$$

Equate coefficients at the same powers in expansions of both sides in (\ref{leo}) and obtain the
system of differential equations
\begin{equation}\frac{da}{dt}=-2\sum_{s=1}^{n-1}e^{-s(t+iu)}A^s(t)a(t),\;\;\;a(0)=a^0,\;\;\;
a_1(t)=1, \label{coe}
\end{equation}
where $A^s$ denotes the $s$-th power of the matrix $A(t)$. To realize the maximum principle we
introduce an adjoint vector $\Psi(t)$,
$$
\Psi(t)=\left(\begin{matrix}
\Psi_1(t)\\
\Psi_2(t)\\
\vdots\\
\Psi_n(t)
\end{matrix}\right),
$$
with complex valued coordinates $\Psi_1,\Psi_2,\dots,\Psi_n$, and the real pseudo Hamilton function
\begin{equation}
H(t,a,\overline{\Psi},u)=\re\left\{-2\sum_{s=1}^{n-1}e^{-s(t+iu)}(A^sa)^T\overline{\Psi}\right\},
\label{ham}
\end{equation}
where $\overline{\Psi}$ means the vector with complex conjugate coordinates and the upper index $T$
is the transposition sign. To come to the Hamiltonian formulation we require that the function
$\overline{\Psi}$ satisfies the adjoint system of differential equations
\begin{equation}
\frac{d\overline{\Psi}}{dt}=2\sum_{s=1}^{n-1}e^{-s(t+iu)}(s+1)(A^T)^s\overline{\Psi}. \label{adj}
\end{equation}

Suppose that for the functional $L$ given by (\ref{fun}), $\max\re\{L\}$ on $S$ is delivered by an
extremal function $f\in S$. By (\ref{lim}) and (\ref{leo}), $f$ corresponds to a real analytic
optimal driving function $u$. Equations (\ref{coe}) and (\ref{adj}) with the optimal function $u$
produce the optimal trajectory $(a(t),\overline{\Psi}(t))$ corresponding to $f$ and $L$,
$a(\infty)=a$, $\Psi(\infty)=(0,\lambda_2,\dots,\lambda_n)^T$. The necessary optimal condition
requires that the optimal function $u$ satisfies the Pontryagin maximum principle, i.e., along the
optimal trajectory $(a(t),\overline{\Psi}(t))$ the function $H(t,a,\overline{\Psi},u)$ as a
function of $u$ is maximized by the optimal driving function $u=u(t)$. Hence this optimal $u$
solves the equation
\begin{equation}
H_u(t,a,\overline{\Psi},u)= \re\left\{2i\sum_{s=1}^{n-1}e^{-s(t+iu)}s(A^sa)^T\overline{\Psi}
\right\}=0,\;\;\;t\geq0, \label{max}
\end{equation}
along the optimal trajectory.

Note that a rotation $f(z)\to f_{\beta}(z):=e^{-i\beta}f(e^{i\beta}z)$ of the extremal function $f$
implies the transformation $L\to L_{\beta}$ of the functional $L$ defined by coefficients
$\lambda_2,\dots,\lambda_n$ according to (\ref{fun}). The functional $L_{\beta}$ should be defined
by the coefficients $\nu_k:=e^{i(k-1)\beta}\lambda_k$, $k=2,\dots,n$, because $f_{\beta}$ maximizes
$\re\{L_{\beta}\}$ on $S$, $\re\{L(f)\}=\re\{L_{\beta}(f_{\beta})\}$. The function $f_{\beta}$ has
a representation (\ref{lim}), (\ref{leo}) with the driving function $u(t)-\beta$ provided $u(t)$
corresponds to $f$. Therefore we can assume without loss of generality that the optimal driving
function $u$ satisfies the initial condition $u(0)=\pi$.

The initial value $\overline{\Psi}(0)$ is uniquely determined by $a=a(\infty)$ and
$\lambda_2,\dots,\lambda_n$ from (\ref{fun}). The following lemma was proved earlier
\cite{Prokhorov} in another version.

\begin{lemma}
Let $a(t)$ and $\Psi(t)$, $\Psi(\infty)=(0,\lambda_2,\dots,\lambda_n)^T$, obey systems (\ref{coe})
and (\ref{adj}). Then
\begin{equation}
\overline{\Psi}_k(0)=\sum_{j=1}^{n-k+1}\overline{\lambda}_{j+k-1}ja_j,\;\;\; k=2,\dots,n.
\label{ini}
\end{equation}
\end{lemma}

\begin{proof} Differentiating (\ref{leo}) with respect to $z$ we have
\begin{equation}
\frac{d}{dt}\left(\frac{1}{e^tw'(z,t)}\right)=\frac{2w(2e^{iu}-w)}{e^tw'(z,t)(e^{iu}-w)^2},\;\;\;
w'(z,0)=1,  \label{der}
\end{equation}
where $w'(z,t)$ is a derivative of $w(z,t)$ with respect to $z$. Considering the expansion
$$\frac{z}{e^tw'(z,t)}=\sum_{k=1}^{\infty}q_k(t)z^k,$$ we obtain for
$q(t)=(q_1(t),\dots,q_n(t))^T,$
\begin{equation}
\frac{dq}{dt}=2\sum_{s=1}^{n-1}e^{-s(t+iu)}(s+1)A^sq.  \label{tra}
\end{equation}

We observe that system (\ref{tra}) for $q$ differs from system (\ref{adj}) for $\overline{\Psi}$
only by the transposition sign $T$ at $A^s$. In order to satisfy the condition
$q(\infty)=(0,\overline{\lambda}_2,\dots,\overline{\lambda}_n)^T$ we denote by
$$g(z,t)=\frac{(\overline{\lambda}_nz^2+\dots+\overline{\lambda}_2z^n)f'(z)}{e^tw'(z,t)}=
\sum_{k=2}^{\infty}c_k(t)z^k$$ and see that $g(z,t)$ obeys the same equation (\ref{der}) where
$1/w'(z,t)$ is substituted by $g(z,t)$. Hence, $c(t)=(c_2(t),\dots,c_{n+1}(t))^T$ obeys system
(\ref{tra}) substituting $q(t)$ by $c(t)$. It is evident that
$c(\infty)=(\overline{\lambda}_n,\dots,\overline{\lambda}_2,0)^T$. The difference in the
transposition sign in(\ref{adj}) and (\ref{tra}) implies that $\overline{\Psi}_k(t)=c_{n-k+2}(t)$,
$k=2,\dots,n$. It remains to evaluate $\overline{\Psi}_k(0)=c_{n-k+2}(0)$ for
$$(\overline{\lambda}_nz^2+\dots+\overline{\lambda}_2z^n)f'(z)=\sum_{k=2}^{\infty}c_k(0)z^k.$$
Straightforward calculations lead to (\ref{ini}) and complete the proof.
\end{proof}

\section{Proof of Theorem 1}

{\it Proof of Theorem 1.} Suppose that a function $f(z)=z+a_2z^2+\dots$ maximizes $\re\{L\}$ and
$\re\{M\}$ given by (\ref{fun}), $m=n$, on $S$. In this case $f$ maximizes also
$\re\{(1-\alpha)L+\alpha M\}$ for all $\alpha\in[0,1]$. The function $f$ is represented by
(\ref{lim}), (\ref{leo}) with a real analytic optimal driving function $u$ in (\ref{leo}). An
optimal trajectory $(a(t),\overline{\Psi}(t))$ obeys systems (\ref{coe}) and (\ref{adj}) with the
optimal function $u$. The maximum principle requires that equality (\ref{max}) for the pseudo
Hamilton function $H(t,a,\overline{\Psi},u)$ holds identically with the optimal $u(t)$ along the
optimal trajectory.

Evaluate $$\frac{\partial H_u}{\partial a_k}\frac{da_k}{dt}:=\frac{\partial H_u}{\partial(\re
a_k)}\frac{d(\re a_k)}{dt}+\frac{\partial H_u}{\partial(\im a_k)}\frac{d(\im
a_k)}{dt},\;\;\;k=2,\dots,n.$$

Put $$\tilde H(t,a,\overline{\Psi},u)=-2\sum_{s=1}^{n-1}e^{-s(t+iu)}(A^sa)^T\overline{\Psi},
\,\,\,\re\tilde H(t,a,\overline{\Psi},u)=H(t,a,\overline{\Psi},u).$$ Adding the equalities
$$\frac{\partial H_u}{\partial(\re a_k)}\frac{d(\re a_k)}{dt}=\re\left(\frac{\partial\tilde
H_u}{\partial a_k}\frac{da_k}{dt}\right)\;\;\text{and}\;\;\frac{\partial H_u}{\partial(\im
a_k)}\frac{d(\im a_k)}{dt}=-\im\left(\frac{\partial\tilde H_u}{\partial
a_k}\frac{da_k}{dt}\right)$$ we obtain using (\ref{coe}) and (\ref{adj})
$$\frac{\partial H_u}{\partial a_k}=\re\left(-4i\sum_{s=1}^{n-2}\sum_{j=1}^{n-s-1}
e^{-(s+j)(t+iu)}s\frac{\partial}{\partial a_k}((A^sa)^T)(A^ja)_k\right),$$ where $(A^ja)_k$ is the
$k$-th coordinate of the column vector $A^ja$. This implies that
\begin{equation}
\frac{\partial H_u}{\partial a}\frac{da}{dt}=\re\left(-4i\sum_{s=1}^{n-2}\sum_{j=1}^{n-s-1}
e^{-(s+j)(t+iu)}s(s+1)a^T(A^T)^{s+j}\overline{\Psi}\right). \label{hu1}
\end{equation}

Similarly, for the vector $$\frac{\partial
H_u}{\partial\overline{\Psi}}\frac{d\overline{\Psi}}{dt}$$ with coordinates $$\frac{\partial
H_u}{\partial\overline{\Psi}_k}\frac{d\overline{\Psi}_k}{dt}:=\frac{\partial
H_u}{\partial(\re\overline{\Psi}_k)}\frac{d(\re\overline{\Psi}_k)}{dt}+\frac{\partial
H_u}{\partial(\im\overline{\Psi}_k)}\frac{d(\im\overline{\Psi}_k)}{dt},$$ we have according to
(\ref{hu1})
\begin{equation}
\frac{\partial H_u}{\partial\overline{\Psi}
}\frac{d\overline{\Psi}}{dt}=\re\left(4i\sum_{s=1}^{n-2}\sum_{j=1}^{n-s-1}
e^{-(s+j)(t+iu)}s(s+1)a^T(A^T)^{s+j}\overline{\Psi}\right)=-\frac{\partial H_u}{\partial
a}\frac{da}{dt}. \label{hu2}
\end{equation}

In the same way, for $$H_{ut^m}(t,a,\overline{\Psi},u)=\re\left(2i(-1)^m\sum_{s=1}^{n-1}
e^{-s(t+iu)}s^{m+1}(A^sa)^T\overline{\Psi}\right),\;\;\;m=1,2,\dots,$$ we have
\begin{equation}
\frac{\partial H_{ut^m}}{\partial\overline{\Psi}}\frac{d\overline{\Psi}}{dt}= -\frac{\partial
H_{ut^m}}{\partial a}\frac{da}{dt},\;\;\;m=1,2,\dots\;. \label{hu3}
\end{equation}

Condition (\ref{th1}) of Theorem 1 means that $H(0,a^0,\overline{\Psi}(0),u)$ attains its maximum
on $[0,2\pi]$ at $u=\pi$. The initial value $\overline{\Psi}(0)$ is given by (\ref{ini}) in Lemma
1. However, the optimal function $u$ preserves its extremal properties when $\lambda_k$ are
substituted by $(1-\alpha)\lambda_k+\alpha\mu_k$, $k=2,\dots,n$. In particular, $u=\pi$ is the
maximum point of $H(0,a^0,\overline{\Psi}(\alpha,0),u)$ for the vector $\overline{\Psi}(\alpha,0)$
with coordinates $$\overline{\Psi}_k(\alpha,0):=\sum_{j=1}^{n-k+1}(\overline{\lambda}_{j+k-1}+
\alpha(\overline{\mu}_{j+k-1}-\overline{\lambda}_{j+k-1}))ja_j,\;\;\;k=2,\dots,n,\;\;\;
\alpha\in[0,1].$$ Hence $u(\alpha,0)=\pi$ is a root of the equation
$$H_u(0,a^0,\overline{\Psi}(\alpha,0),u)=
\re\sum_{k=2}^n2i(k-1)e^{-i(k-1)u}\;\overline{\Psi}_k(\alpha,0)=0.$$ So $u(\alpha,0)$ does not
depend on $\alpha$. Elementary calculations show that the equation
$$\frac{\partial u(\alpha,0)}{\partial\alpha}=0,\;\;\;0\leq\alpha\leq1,$$ is equivalent to
(\ref{th3}).

Changing the initial value in system (\ref{adj}) from $\overline{\Psi}(0)$ to
$\overline{\Psi}(\alpha,0)$ we preserve the function $f$ and the optimal driving function
$u=u(t)=u(\alpha,t)$ in (\ref{leo}) but the adjoint coordinate $\overline{\Psi}(t)$ in the optimal
trajectory $(a(t),\overline{\Psi}(t))$ is substituted by $\overline{\Psi}(\alpha,t)$. Differentiate
$H_u(t,a,\overline{\Psi},u)$ in (\ref{max}) with respect to $t$ along the optimal trajectory
$(a(t),\overline{\Psi}(\alpha,t))$. Taking into account (\ref{hu2}) we obtain
\newpage
\begin{equation}
\frac{d}{dt}H_u(t,a(t),\overline{\Psi}(\alpha,t),u(\alpha,t))=\frac{\partial H_u}{\partial
t}+\frac{\partial H_u}{\partial a}\frac{da}{dt}+\frac{\partial
H_u}{\partial\overline{\Psi}}\frac{\partial\overline{\Psi}}{\partial t}+\frac{\partial
H_u}{\partial u}\frac{\partial u}{\partial t}= \label{di1}
\end{equation}
$$H_{ut}(t,a(t),\overline{\Psi}(\alpha,t),u(\alpha,t))+
H_{uu}(t,a(t),\overline{\Psi}(\alpha,t),u(\alpha,t))u_t(\alpha,t)=0,$$ which gives the formula
$$u_t(\alpha,0)=-\frac{H_{ut}(0,a^0,\overline{\Psi}(\alpha,0),\pi)}
{H_{uu}(0,a^0,\overline{\Psi}(\alpha,0),\pi)}=$$
$$-\;\frac{\sum_{k=2}^n\sum_{j=1}^{n-k+1}(-1)^k(k-1)^2\im(\overline{\lambda}_{j+k-1}+\alpha
(\overline{\mu}_{j+k-1}-\overline{\lambda}_{j+k-1})ja_j)}
{\sum_{k=2}^n\sum_{j=1}^{n-k+1}(-1)^k(k-1)^2\re(\overline{\lambda}_{j+k-1}+\alpha
(\overline{\mu}_{j+k-1}-\overline{\lambda}_{j+k-1})ja_j)}\;,$$ where the denominator does not
vanish because of (\ref{th2}). Condition (\ref{th4}) for $m=1$ implies that for $u(t)=u(\alpha,t)$,
\begin{equation}
u'(0)=u_t(\alpha,0)=0. \label{di3}
\end{equation}

Suppose by induction that for $u=u(t)=u(\alpha,t)$,
\begin{equation}
u^{(p)}(0)=u_{t^p}(\alpha,0)=0,\;\;\;p=1,\dots,m-1, \label{di4}
\end{equation}
and differentiate $H_u(t,a(t),\overline{\Psi}(\alpha,t),u(\alpha,t))$ along the optimal trajectory
$m-1$ times. Taking into account (\ref{hu3}) - (\ref{di4}) and the inductive formula
$$\frac{d^{m-1}}{dt^{m-1}}H_u(t,a(t),\overline{\Psi}(\alpha,t),u(\alpha,t))=
H_{ut^{m-1}}(t,a(t),\overline{\Psi}(\alpha,t),u(\alpha,t))+$$
$$\sum_{j=1}^{m-2}R_j(t)u_{t^j}(\alpha,t)+H_{uu}(t,a(t),\overline{\Psi}(\alpha,t),u(\alpha,t))
u_{t^{m-1}}(\alpha,t)$$ with inductively evaluated functions $R_j(t)$, $j=1,\dots,m-2$, we obtain
$$\frac{d^m}{dt^m}H_u(t,a(t),\overline{\Psi}(\alpha,t),u(\alpha,t))_{t=0}=H_{ut^m}(0,a^0,
\overline{\Psi}(\alpha,0),\pi)+$$
$$H_{uu}(0,a^0,\overline{\Psi}(\alpha,0),\pi)u_{t^m}(\alpha,0)=0.$$ This allows us to find
$u_{t^m}(\alpha,0)$,
$$u_{t^m}(\alpha,0)=-\frac{H_{ut^m}(0,a^0,\overline{\Psi}(\alpha,0),\pi)}
{H_{uu}(0,a^0,\overline{\Psi}(\alpha,0),\pi)}=$$
$$-\;\frac{\sum_{k=2}^n\sum_{j=1}^{n-k+1}(-1)^k(k-1)^{m+1}\im(\overline{\lambda}_{j+k-1}+\alpha
(\overline{\mu}_{j+k-1}-\overline{\lambda}_{j+k-1})ja_j)}
{\sum_{k=2}^n\sum_{j=1}^{n-k+1}(-1)^k(k-1)^2\re(\overline{\lambda}_{j+k-1}+\alpha
(\overline{\mu}_{j+k-1}-\overline{\lambda}_{j+k-1})ja_j)}\;,$$ where the denominator does not
vanish because of (\ref{th2}). Conditions (\ref{th4}) imply that for $u(t)=u(\alpha,t)$,
\begin{equation}
u^{(m)}(0)=u_{t^m}(\alpha,0)=0. \label{di5}
\end{equation}

Thus it follows from (\ref{di3}) - (\ref{di5}) that $u(0)=\pi$ and $u^{(m)}(0)=0$ for all
$m=1,2,\dots\;.$ Since $u(t)$ is real analytic, $u(t)=\pi$ identically for $t\geq0$. This driving
function $u$ determines the Koebe function $k(z)$ by (\ref{lim}), (\ref{coe}) which completes the
proof of Theorem 1.

\begin{remark}
Condition (\ref{th1}) of Theorem 1 can be achieved with the help of a suitable rotation of the
extremal function $f$.
\end{remark}

\begin{remark}
Condition (\ref{th2}) of Theorem 1 is not essential. Indeed, the optimal function $u(t)$ maximizes
$H(t,a,\overline{\Psi},u)$ along the optimal trajectory, and $u(0)=\pi$ is the maximum point of the
function $H(0,a^0,\overline{\Psi}(0),u)$. Therefore
$$H_u(0,a^0,\overline{\Psi}(0),\pi)=0\;\;\text{and}\;\;H_{uu}(0,a^0,\overline{\Psi}(0),\pi)\leq0.$$
If $H_{uu}(0,a^0,\overline{\Psi}(0),\pi)=0$, then there is an even number $m=2l$, $l>1$, such that
$$H_{u^q}(0,a^0,\overline{\Psi}(0),\pi)=0,\;\;2\leq q\leq 2l-1,\;\;
H_{u^{2l}}(0,a^0,\overline{\Psi}(0),\pi)<0.$$ Since $H(t,a,\overline{\Psi},u)$ is linear with
respect to $\overline{\Psi}$, it is possible to choose a minimal even number which provides the
above property for all $$\overline{\Psi}_k(0)=\sum_{j=1}^{n-k+1}(\overline{\lambda}_{j+k-1}+
\alpha(\overline{\mu}_{j+k-1}-\overline{\lambda}_{j+k-1})),\;\;k=2,\dots,n,\;\;0\leq\alpha\leq1,$$
with certain $\lambda_2,\dots,\lambda_n,\mu_2,\dots,\mu_n$. In this case we repeat the proof of
Theorem 1 changing $H_{uu}(0,a^0,\overline{\Psi}(0),u)$ for
$H_{u^{2p}}(0,a^0,\overline{\Psi}(0),u)$.
\end{remark}

\begin{corollary}
Let a function $f(z)\in S$ with real coefficients $a_2,\dots,a_n$ maximize $\re\{L\}$ and
$\re\{M\}$ for functionals $L$ and $M$ given by (\ref{fun}), $m=n$, with real coefficients
$\lambda_2,\dots,\lambda_n$ and $\mu_2,\dots,\mu_n$ and satisfy the conditions of Theorem 1. Then
$f(z)$ is the Koebe function $k(z)$.
\end{corollary}

\begin{proof}
Under the conditions of Corollary 1 the sufficient conditions (\ref{th4}) are trivially verified.
\end{proof}

\section{Necessary conditions in the two-functional conjecture}

Observe that equality (\ref{th3}) is necessary under the conditions of Theorem 1. Let us adduce
another necessary relations which are not too far from the sufficient conditions (\ref{th4}) but
lead to the opposite conclusions.

\begin{theorem}
Let a function $f(z)\in S$ and functionals $L$ and $M$ satisfy the conditions of Theorem 1. Then
either conditions (\ref{th4}) are satisfied for all $m\geq2$ or there are numbers $m\geq1$ and
$c_m\neq0$ such that conditions (\ref{th4}) are satisfied for all $p<m$ and
\begin{equation}
\im\sum_{k=2}^n\sum_{j=1}^{n-k+1}(-1)^k(k-1)^{m+1}((1-\alpha)\overline{\lambda}_{j+k-1}+
\alpha\overline{\mu}_{j+k-1})= \label{2th}
\end{equation}
$$c_m\re\sum_{k=2}^n\sum_{j=1}^{n-k+1}(-1)^k(k-1)^2((1-\alpha)\overline{\lambda}_{j+k-1}+
\alpha\overline{\mu}_{j+k-1}),\;\;\;0\leq\alpha\leq1.$$ In the last case $f$ is not a rotation of
the Koebe function $k(z)$.
\end{theorem}

\begin{proof}
Suppose that conditions (\ref{th4}) for $f$, $L$ and $M$ are satisfied for all $p<m$. It was shown
in the proof of Theorem 1 that in this case $u^{(p)}(0)=0$, $p=1,\dots,m-1$, and
\begin{equation}
u^{(m)}(0)=u_{t^m}(\alpha,0)=-\;\frac{H_{ut^m}(0,a^0,\overline{\Psi}(\alpha,0),\pi)}{H_{uu}(0,a^0,
\overline{\Psi}(\alpha,0),\pi)}= \label{3th}
\end{equation}
$$(-1)^m\frac{\im\sum_{k=2}^n\sum_{j=1}^{n-k+1}(-1)^k(k-1)^{m+1}((1-\alpha)
\overline{\lambda}_{j+k-1}+\alpha\overline{\mu}_{j+k-1})}
{\re\sum_{k=2}^n\sum_{j=1}^{n-k+1}(-1)^k(k-1)^2((1-\alpha)\overline{\lambda}_{j+k-1}+
\alpha\overline{\mu}_{j+k-1})}.$$

As soon as $u_{t^m}(\alpha,0)$ does not depend on $\alpha\in[0,1]$, we conclude that either the
numerator in (\ref{3th}) vanishes or the numerator and the denominator are proportional. The last
case is reflected in condition (\ref{2th}) which means that $u_{t^m}(\alpha,0)$ is equal to
$(-1)^mc_m\neq0$. Therefore $u(t)$ does not correspond to any rotation of the Koebe function, and
this completes the proof.
\end{proof}

Note that condition (\ref{2th}) is reduced to (\ref{th4}) if $c_m=0$. The two-functional conjecture
supposes that $c_m=0$ every time.

Write the proportionality condition in (\ref{3th}) in another way. Equality (\ref{2th}) is
equivalent to the system of equations
$$\im\sum_{k=2}^n\sum_{j=1}^{n-k+1}(-1)^k(k-1)^{m+1}\overline{\lambda}_{j+k-1}ja_j=
c_m\re\sum_{k=2}^n\sum_{j=1}^{n-k+1}(-1)^k(k-1)^2\overline{\lambda}_{j+k-1}ja_j,$$
$$\im\sum_{k=2}^n\sum_{j=1}^{n-k+1}(-1)^k(k-1)^{m+1}\overline{\mu}_{j+k-1}ja_j=
c_m\re\sum_{k=2}^n\sum_{j=1}^{n-k+1}(-1)^k(k-1)^2\overline{\mu}_{j+k-1}ja_j.$$

\begin{example}
Let $$L=\lambda a_2+a_4,\;\;\;\lambda\in\Bbb R.$$ For $\lambda\geq0$, $\re\{L\}$ is maximized on
$S$ by the Koebe function $k(z)$. We will be looking for negative $\lambda$ so that $\re\{L\}$ is
still locally maximized by $k(z)$. To realize the L\"owner approach we apply Lemma 1 and put
$$\Psi_2(0)=9+\lambda,\;\;\;\Psi_3(0)=4,\;\;\;\Psi_4(0)=1$$ in (\ref{coe}), (\ref{adj}). Then
$$H(0,a^0,\overline{\Psi}(0),u)=-2(\cos3u+4\cos2u+(9+\lambda)\cos u):=p_{\lambda}(u).$$ Denote by
$\lambda_0=-0.931...$ the root of the equation
\begin{equation}
25\lambda^3+37\lambda^2+16\lambda+3=0. \label{lam}
\end{equation}
Straightforward calculations show that for $\lambda>\lambda_0$,
$$\max_{u\in[0,2\pi]}p_{\lambda}(u)=p_{\lambda}(\pi),$$ and this property is preserved for a slight
variation of coefficients of $p_{\lambda}(u)$. Besides, the choice of real initial value
$\overline{\Psi}(0)$ implies that $a(t)$ and $\overline{\Psi}(t)$ remain to be real along the
optimal trajectory $(a(t),\overline{\Psi}(t))$, and $H(t,a(t),\overline{\Psi}(t),u)$ is a cubic
polynomial with respect to $\cos u$. Therefore the real analytic optimal driving function $u(t)$ is
equal to $\pi$ identically for $t>0$ small enough. This implies that $u(t)=\pi$ for all $t\geq0$.
It is verified that if $\lambda<\lambda_0$, then $$\max_{u\in[0,2\pi]}p_{\lambda}(u)\neq
p_{\lambda}(\pi).$$
\end{example}

So we have proved the following proposition.
\begin{proposition}
Let $\lambda_0=-0.931...$ be the root of equation (\ref{lam}). For $\lambda\geq\lambda_0$, the
Koebe function $k(z)$ locally maximizes $\re\{L\}:=\re\{\lambda a_2+a_4\}$ on $S$. For
$\lambda<\lambda_0$, $k(z)$ does not maximize $\re\{L\}$ on $S$.
\end{proposition}

\end{document}